\documentclass[10pt]{article}

\title{Injective types in univalent mathematics}

\author{Mart\'{\i}n H\"otzel Escard\'o \\ School of Computer Science \\ University of Birmingham, UK}

\usepackage{amsmath,stmaryrd,amssymb,amsthm,enumerate,url,hyperref,scalerel,mathtools}

\usepackage[bb=boondox]{mathalfa} 

\usepackage[nohug]{diagrams}
\newarrow{Into}{C}{-}{-}{-}{>}
\newarrow{Onto}{-}{-}{-}{-}{>>}
\newarrow{Eto}{.}{.}{.}{.}{>}
\newarrow{Id}{=}{=}{=}{=}{=}
\newarrow{Mapsto}{|}{-}{-}{-}{>}

\usepackage{xcolor}
\definecolor{black}{rgb}{0,0,0.0}

\definecolor{darkblue}{rgb}{0,0,0.4}
\newcommand{\db}{\textcolor{darkblue}}
\definecolor{darkgreen}{rgb}{0,0.3,0.9}

\definecolor{grey}{rgb}{0.4,0.4,0.4}

\definecolor{darkred}{rgb}{0.5,0,0}

\definecolor{links}{rgb}{0.47, 0.27, 0.23}

\hypersetup{
    colorlinks, linkcolor={links},
    citecolor={links}, urlcolor={links}
}

\newcommand{\emb}{\hookrightarrow}
\newcommand{\partialto}{\rightharpoonup}

\makeatletter
\newcommand\leftcolor[2]{%
  \edef\@tempa{#1}%
  \ifx\@tempa\@empty
    \edef\@tempa{darkblue}%
  \fi
 \mathinner\bgroup\begingroup\color{\@tempa}\left#2\normalcolor%
}
\newcommand\rightcolor[2]{%
  \edef\@tempa{#1}%
  \ifx\@tempa\@empty
    \edef\@tempa{darkblue}%
  \fi
  \color{\@tempa}\right#2\endgroup\egroup%
}
\makeatother

\newcommand{\bracket}[1]{\leftcolor{darkgreen}(\,#1\,\rightcolor{darkgreen})}

\newcommand{\trunc}[1]{\leftcolor{darkblue}\llbracket\,\db{#1}\,\rightcolor{darkblue}\rrbracket}

\newcommand{\df}[1]{\emph{\db{#1}}}

\newcommand{\m}[1]{\db{$#1$}}
\newcommand{\M}[1]{\[\db{#1}\]}

\newcommand{\EM}{\operatorname{EM}}
\newcommand{\pcomp}{\operatorname{\bullet}}
\newcommand{\transport}{\operatorname{transport}}
\newcommand{\ap}{\operatorname{ap}}
\newcommand{\fst}{\operatorname{pr}_1}
\newcommand{\snd}{\operatorname{pr}_2}
\newcommand{\lift}{\operatorname{lift}}
\newcommand{\id}{\operatorname{id}}
\newcommand{\comp}{\mathrel{\circ}}

\newcommand{\N}{\mathbb{N}}

\newcommand{\U}{\mathcal{U}}
\newcommand{\V}{\mathcal{V}}
\newcommand{\W}{\mathcal{W}}
\newcommand{\T}{\mathcal{T}}
\newcommand{\Lift}{\mathcal{L}}

\newcommand{\Zero}{\mathbb{0}}
\newcommand{\One}{\mathbb{1}}
\newcommand{\one}{\operatorname{\star}}
\newcommand{\Id}{\operatorname{Id}}
\newcommand{\refl}{\operatorname{refl}}
\newcommand{\eqdef}{\overset{\text{def}}{=}}

\newcommand{\idto}{=}

\newcommand{\edown}{\scaleobj{0.7}{\,\downarrow\,}}
\newcommand{\eup}{\scaleobj{0.7}{\,\uparrow\,}}
\newcommand{\wps}{\mathrel{\,\,\preceq\,\,}}

\newtheorem{numbered}{}
\swapnumbers
\newtheorem{theorem}[numbered]{Theorem}
\newtheorem{lemma}[numbered]{Lemma}
\newtheorem{remark}[numbered]{Remark}
\newtheorem{proposition}[numbered]{Proposition}
\newtheorem{corollary}[numbered]{Corollary}
\theoremstyle{definition}
\newtheorem{definition}[numbered]{Definition}

\begin{document}

\maketitle

\begin{abstract}
  We investigate the injective types and the algebraically injective
  types in univalent mathematics, both in the absence and in the
  presence of propositional resizing. Injectivity is defined by the
  surjectivity of the restriction map along any embedding, and
  algebraic injectivity is defined by a given section of the
  restriction map along any embedding. Under propositional resizing
  axioms, the main results are easy to state: (1)~Injectivity is
  equivalent to the propositional truncation of algebraic injectivity.
  (2)~The algebraically injective types are precisely the retracts of
  exponential powers of universes. (2a)~The algebraically injective
  sets are precisely the retracts of powersets. (2b)~The algebraically
  injective \m{(n+1)}-types are precisely the retracts of exponential
  powers of universes of \m{n}-types. (3)~The algebraically injective
  types are also precisely the retracts of algebras of the partial-map
  classifier. From~(2) it follows that any universe is embedded as a
  retract of any larger universe.  In the absence of propositional
  resizing, we have similar results that have subtler statements which
  need to keep track of universe levels rather explicitly, and are
  applied to get the results that require resizing.


\end{abstract}

\section{Introduction}

We investigate the injective types and the algebraically injective
types in univalent mathematics, both in the absence and in the
presence of propositional resizing axioms. These notions of
injectivity are about the extension problem
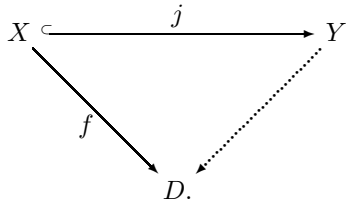
\begin{diagram}
  X & & \rInto^j & & Y  \\
  & \rdTo_f &  & \ldEto & \\
  & & D. & &
\end{diagram}
The injectivity of a type \m{D:\U} is defined by the surjectivity
of the restriction map \m{(-) \comp j} along any embedding~\m{j}:
\M{ \Pi(X,Y : \U)\, \Pi (j : X \emb Y)\, \Pi(f : X \to D)\, \exists (g :
  Y \to D)\, g \comp j = f,
}
so that we get an \emph{unspecified} extension~\m{g} of~\m{f}
along~\m{j}.
The algebraic injectivity of \m{D} is defined by a given section
\m{(-) \mid j} of the restriction map \m{(-) \comp j}, following Bourke's
terminology~\cite{bourke:2017}. By \m{\Sigma{-}\Pi}-distributivity,
this amounts to
\M{
  \Pi(X,Y : \U)\,
  \Pi (j : X \emb Y)\, \Pi(f : X \to D)\, \Sigma (f \mid j : Y \to D),
  f \mid j \comp j = f,
}
so that we get a \emph{designated} extension~\m{f \mid j} of~\m{f}
along~\m{j}.  Formally, in this definition, \m{f \mid j} can be
regarded as a variable, but we instead think of the symbol ``\m{\mid}'' as
a binary operator.

For the sake of generality, we work without assuming or rejecting the
principle of excluded middle, and hence without assuming the axiom of
choice either. Moreover, we show that the principle of excluded middle
holds if and only if all pointed types are algebraically injective,
and, assuming resizing, if and only if all inhabited types are
injective, so that there is nothing interesting to say about
(algebraic) injectivity in its presence. That pointness and
inhabitedness are needed is seen by considering the embedding \m{\Zero
  \emb \One}.

Under propositional resizing principles~\cite{hottbook}
(Definitions~\ref{resizing} and~\ref{omega:resizing} below), the main
results are easy to state:
\begin{enumerate}
\item Injectivity is equivalent to the propositional truncation of
  algebraic injectivity.

  (This can be seen as a form of choice that
  just holds, as it moves a propositional truncation inside a
  \m{\Pi}-type to outside the \m{\Pi}-type, and may be related to
  \cite{kenney:2011}.)
   \item The algebraically injective types are precisely the retracts of
     exponential powers of universes. Here by an exponential power of a type \m{B} we mean a type of the form \m{A \to B}, also written \m{B^A}.

     In particular,
       \begin{enumerate}
       \item The algebraically injective sets are precisely the
         retracts of powersets.

       \item The algebraically injective \m{(n+1)}-types are precisely
         retracts of exponential powers of the universes of \m{n}-types.
       \end{enumerate}
     Another consequence is that any universe is embedded as a retract of any
     larger universe.
   \item The algebraically injective types are also precisely the
       underlying objects of the algebras of the partial-map
       classifier.
\end{enumerate}
In the absence of propositional resizing, we have similar results
that have subtler statements that need to keep track of universe
levels rather explicitly.
Most constructions developed in this paper are in the absence of
propositional resizing. We apply them, with the aid of a notion of
algebraic flabbiness, which is related to the partial-map classifier,
to derive the results that rely on resizing mentioned above.

\paragraph{Acknowledgements.} Mike Shulman has acted as a sounding
board over the years, with many helpful remarks, including in
particular the suggestion of the terminology \emph{algebraic
  injectivity} from~\cite{bourke:2017} for the notion we consider
here.

\section{Underlying formal system} \label{foundations}

Our handling of universes has a model in \m{\infty}-toposes following
Shulman~\cite{2019arXiv190407004S}. It differs from that of the HoTT
book~\cite{hottbook}, and Coq~\cite{coq}, in that we don't assume
cumulativity, and it agrees with that of Agda~\cite{agda}.

\subsection{Our univalent type theory}

Our underlying formal system can be considered to be a subsystem of that
used in UniMath~\cite{unimath}.
\begin{enumerate}
\item We work within an intensional Martin-L\"of type theory with
  types \m{\Zero} (empty type), \m{\One} (one-element type with
  \m{\one:\One}), \m{\N} (natural numbers), and type formers \m{+}
  (binary sum), \m{\Pi} (product), \m{\Sigma} (sum) and \m{\Id}
  (identity type), and a hierarchy of type universes ranged over by
  \m{\U,\V,\W,\T}, closed under them in a suitable sense discussed
  below.

  We take these as required closure properties of our formal system,
  rather than as an inductive definition.

\item We assume a universe \m{\U_0}, and for each universe \m{\U} we
  assume a successor universe \m{\U^+} with \m{\U : \U^+}, and for any
  two universes \m{\U,\V} a least upper bound \m{\U \sqcup \V}. We
  stipulate that we have \m{\U_0 \sqcup \U = \U} and \m{\U \sqcup \U^+
    = \U^+} definitionally, and that the operation \m{(-)\sqcup(-)} is
  definitionally idempotent, commutative, and associative, and that
  the successor operation \m{(-)^+} distributes over \m{(-)\sqcup(-)}
  definitionally.

\item We don't assume that the universes are cumulative on the nose,
  in the sense that from \m{X : \U} we would be able to deduce that
  \m{X : \U \sqcup \V} for any \m{\V}, but we also don't assume that
  they are not. However, from the assumptions formulated below, it
  follows that for any two universes \m{\U,\V} there is a map
  \m{\lift_{\U,\V} : \U \to \U \sqcup \V}, for instance \m{X \mapsto X
    + \Zero_\V}, which is an embedding with \m{\lift X \simeq X} if
  univalence holds (we cannot write the identity type \m{\lift X = X},
  as the left- and right-hand sides live in the different types \m{\U}
  and \m{\U \sqcup \V}, which are not (definitionally) the same in
  general).

\item We stipulate that we have copies \m{\Zero_\U} and \m{\One_\V} of the
  empty and singleton types in each universe \m{\U} (with the subscripts
  often elided).
\item We stipulate that if \m{X : \U} and \m{Y : \V}, then \m{X+Y : \U \sqcup \V}.
\item We stipulate that if \m{X : \U} and \m{A : X \to \V} then
  \m{\Pi_X A : \U \sqcup \V}. We abbreviate this product type as \m{\Pi A}
  when \m{X} can be inferred from \m{A}, and sometimes we write it
  verbosely as \m{\Pi (x:X), A \, x}.

  In particular, for types \m{X : \U} and \m{Y : \V}, we have the function
  type \m{X \to Y : \U \sqcup \V}.
\item The same type stipulations as for \m{\Pi}, and the same
  grammatical conventions apply to the sum type former \m{\Sigma}.

  In particular, for types \m{X : \U} and \m{Y : \V}, we have the cartesian product \m{X \times Y : \U \sqcup \V}.

\item We assume the \m{\eta} rules for \m{\Pi} and \m{\Sigma}, namely
  that \m{f = \lambda x, f \, x} holds definitionally for any \m{f} in
  a \m{\Pi}-type and that \m{z=(\fst z , \snd z)} holds definitionally
  for any \m{z} in a \m{\Sigma} type, where \m{\fst} and \m{\snd} are
  the projections.

\item For a type \m{X} and points \m{x,y:X}, the identity type \m{\Id_{X} x \, y} is abbreviated as \m{\Id x \, y} and often written \m{x \idto_X y} or simply \m{x \idto y}.

  The elements of the identity type \m{x=y} are called identifications
  or paths from \m{x} to~\m{y}.

\item When making definitions, definitional equality is written ``$\eqdef$''. When it is invoked, it is written e.g.\ ``\m{x = y} definitionally''. This is consistent with the fact that any definitional equality \m{x = y} gives rise to an element of the identity type \m{x = y} and should therefore be unambiguous.

\item When we say that something is the case by construction, this means we
  are expanding definitional equalities.

\item We tacitly assume univalence~\cite{hottbook}, which gives
  function extensionality (pointwise equal functions are equal) and
  propositional extensionality (logically equivalent subsingletons are
  equal).

\item We work with the existence of propositional, or subsingleton, truncations as an
  assumption, also tacit.  The HoTT book~\cite{hottbook}, instead,
  defines type formation \df{rules} for propositional truncation as a
  syntactical construct of the formal system. Here we take
  propositional truncation as an axiom for any pair of
  universes \m{\U,\V}: \db{
    \begin{align*}
      \Pi (X:\U)\,  \Sigma & (\trunc{X}  : \U), \\
      & \mathrel{\phantom{\times}} \text{\m{\trunc{X}} is a proposition} \times (X \to \trunc{X}) \\
  & \times \bracket{\Pi (P : \V), \text{\m{P} is a proposition} \to (X \to P) \to \trunc{X} \to P}.
    \end{align*}}

  We write \m{\mid x \mid} for the insertion of \m{x:X} into the type
  \m{\trunc{X}} by the assumed function \m{X \to \trunc{X}}.  We also
  denote by \m{\bar{f}} the function \m{\trunc{X} \to P} obtained by
  the given ``elimination rule'' \m{(X \to P) \to \trunc{X} \to P}
  applied to a function \m{f:X \to P}. The universe \m{\U} is that of
  types we truncate, and \m{\V} is the universe where the propositions
  we eliminate into live.  Because the existence of propositional
  truncations is an assumption rather than a type formation rule, its
  so-called ``computation'' rule \M{\bar{f} \mid x \mid = f x} doesn't
  hold definitionally, of course, but is established as a derived
  identification, by the definition of proposition.

\end{enumerate}

\subsection{Terminology and notation}
\label{existence:terminology}

We assume that the readers are already familiar with the notions of
univalent mathematics, e.g.\ from the HoTT book~\cite{hottbook}. The
purpose of this section is to establish terminology and notation only,
particularly regarding our modes of expression that diverge from the HoTT
book.
\begin{enumerate}
\item A type \m{X} is a singleton, or contractible, if there is a
  designated \m{c:X} with \m{x = c} for all \m{x:X}:
  \M{
    \text{\m{X} is a singleton} \eqdef \Sigma (c : X), \Pi (x:X), x = c.
  }

\item A proposition, or subsingleton, or truth value, is a type with
  at most one element, meaning that any two of its elements are
  identified:
  \M{
  \text{\m{X} is a proposition} \eqdef \Pi(x,y:X), x=y.
  }

\item
  By an unspecified element of a type \m{X} we mean a (specified)
  element of its propositional truncation~\m{\trunc{X}}.

  We say that a type is inhabited if it has an unspecified element.

  If the type \m{X} codifies a mathematical statement, we say that
  \m{X} holds in an unspecified way to mean the assertion
  \m{\trunc{X}}. For example, if we say that the type \m{A} is a
  retract of the type \m{B} in an unspecified way, what we mean is that
  \m{\trunc{\text{\m{A} is a retract of \m{B}}}}.

\item Phrases such as ``there exists'', ``there is'', ``there is some'', ``for some'' etc.\ indicate a propositionally truncated \m{\Sigma}, and symbolically we write  \M{(\exists (x:X), A \, x) \eqdef  \trunc{\Sigma (x:X), A \, x}.}
For emphasis, we may say that there is an unspecified \m{x:X} with \m{A\,x}.

  When the meaning of existence is intended to be (untruncated)
  \m{\Sigma}, we use phrases such as ``there is a designated'', ``there
  is a specified'', ``there is a distinguished'', ``there is a given'', ``there is a chosen'', ``for some chosen'', ``we can find'' etc.

  The statement that there is a unique \m{x:X} with \m{A \, x} amounts to
  the assertion that the type \m{\Sigma (x:X), A \, x} is a singleton:
  \M{
    (\exists! (x:X), A \, x) \eqdef \text{the type \m{\Sigma (x:X), A \, x} is a singleton}.
  }
  That is, there is a unique pair \m{(x,a)} with \m{x:X} and \m{a : A\,
  x}. This doesn't need to be explicitly propositionally truncated,
  because singleton types are automatically propositions.

  The statement that there is at most one \m{x:X} with \m{A \, x}
  amounts to the assertion that the type \m{\Sigma (x:X), A \, x} is a
  subsingleton (so we have at most one pair \m{(x,a)} with \m{x:X} and
  \m{a : A\, x}).

\item We often express a type of the form \m{\Sigma(x:X), A \, x} by
  phrases such as ``the type of \m{x:X} with \m{A \, x}''.

  For example, if we define the fiber of a point \m{y:Y} under a
  function \m{f : X \to Y} to be the type \m{f^{-1}(y)} of points \m{x:X}
  that are mapped by \m{f} to a point identified with \m{y}, it
  should be clear from the above conventions that we mean
  \M{
   f^{-1}(y) \eqdef \Sigma (x : X), f x = y.
  }
 Also, with the above terminological conventions, saying that the
 fibers of \m{f} are singletons (that is, that \m{f} is an equivalence)
 amounts to the same thing as saying that for every \m{y:Y} there is a
 unique \m{x:X} with \m{f(x)=y}.

 Similarly, we say that such an \m{f} is an embedding if for every
 \m{y:Y} there is at most one \m{x:X} with \m{f(x)=y}. In passing, we
 remark that, in general, this is stronger than \m{f} being
 left-cancellable, but coincides with left-cancellability if the type
 \m{Y} is a set (its identity types are all subsingletons).

\item We sometimes use the mathematically more familiar ``maps to''
  notation~\m{\mapsto} instead of type-theoretical lambda notation
  \m{\lambda} for defining nameless functions.

\item Contrarily to an existing convention among some practitioners,
  we will not reserve the word \df{is} for mathematical statements
  that are subsingleton types. For example, we say that a type is
  algebraically injective to mean that it comes equipped with suitable
  data, or that a type \m{X} is a retract of a type \m{Y} to mean that
  there are designated functions \m{s : X \to Y} and \m{r : Y \to X},
  and a designated pointwise identification \m{r \comp s \sim \id}.

\item Similarly, we don't reserve the words \df{theorem}, \df{lemma},
  \df{corollary} and \df{proof} for constructions of elements of
  subsingleton types, and all our constructions are indicated by the
  word proof, including the construction of data or structure.

  Because \df{proposition} is a semantical rather than syntactical
  notion in univalent mathematics, we often have situations when we
  know that a type is a proposition only much later in the
  mathematical development. An example of this is univalence. To know
  that this is a proposition, we first need to state and prove many
  lemmas, and even if these lemmas are propositions themselves, we
  will not know this at the time they are stated and proved. For
  instance, knowing that the notion of being an equivalence is a
  proposition requires function extensionality, which follows from
  univalence. Then this is used to prove that univalence is a
  proposition.

\end{enumerate}

\subsection{Formal development}

A computer-aided formal development of the material of this paper has
been performed in Agda~\cite{agda}, occasionally preceded by pencil
and paper scribbles, but mostly directly in the computer with the aid
of Agda's interactive features. This paper is an unformalization of
that development. We emphasize that not only numbered statements in
this paper have formal counterparts, but also the comments in passing,
and that the formal version has more information than what we
choose to report here.

We have two versions. One of them~\cite{injective:blackboard} is in
\df{blackboard style}, with the ideas in the order they have come to
our mind over the years, in a fairly disorganized way, and with local
assumptions of univalence, function extensionality, propositional
extensionality and propositional truncation. The other
one~\cite{injective:article} is in \df{article style}, with univalence
and existence of propositional truncations as global assumptions, and
functional and propositional extensionality derived from
univalence. This second version follows closely this paper (or rather
this paper follows closely that version), organized in a way more
suitable for dissemination, repeating the blackboard definitions, in a
definitionally equal way, and reproducing the proofs and constructions
that we consider to be relevant while invoking the blackboard for the
routine, unenlightening ones. The blackboard version also has
additional information that we have chosen not to include in the
article version of the Agda development or this paper.

An advantage of the availability of a formal version is that, whatever
steps we have omitted here because we considered them to be
obvious or routine, can be found there, in case of doubt.


\section{Injectivity with universe levels}

As discussed in the introduction, in the absence of propositional
resizing we are forced to keep track of universe levels rather
explicitly.
\begin{definition}
We say that a type \m{D} in a universe \m{\W} is \df{\m{\U,\V}-injective}
to mean
\M{ \Pi(X : \U)\, \Pi(Y : \V)\, \Pi (j : X \emb Y)\, \Pi(f : X \to D),\, \exists (g :
  Y \to D), g \comp j \sim f,
}
and that it is \df{algebraically \m{\U,\V}-injective} to mean
\M{
  \Pi(X : \U) \,\Pi(Y : \V)\,
  \Pi (j : X \emb Y) \, \Pi(f : X \to D) ,\, \Sigma (f \mid j : Y \to D),
  f \mid j \comp j \sim f.
}
\end{definition}
\noindent Notice that, because we have function extensionality, pointwise
equality~\m{\sim} of functions is equivalent to equality, and
hence equal to equality by univalence. But it is more convenient for
the purposes of this paper to work with pointwise equality in these
definitions.

\section{The algebraic injectivity of universes}

Let \m{\U,\V,\W} be universes, \m{X:\U} and \m{Y : \V} be types, and
\m{f : X \to \W} and \m{j : X \to Y} be given functions, where \m{j} is not necessarily an embedding.
  We define functions \m{f \edown j} and \m{f \eup j} of type \m{Y \to \U \sqcup \V \sqcup \W}
  by
  \db{\begin{eqnarray*}
    (f \edown j) \, y & \eqdef & \Sigma (w : j^{-1}(y)), f(\fst w), \\
    (f \eup j) \, y & \eqdef & \Pi (w : j^{-1}(y)), f(\fst w).
  \end{eqnarray*}}

\begin{lemma}
    If \m{j} is an embedding, then both \m{f \edown j} and \m{f \eup j} are extensions of \m{f} along~\m{j} up to equivalence, in the sense that \M{(f \edown j \comp j) \, x \simeq f x \simeq (f \eup j \comp j) \, x,}
and hence extensions up to equality if \m{\W} is taken to be \m{\U \sqcup \V}, by univalence.
\end{lemma}
\noindent Notice that if \m{\W} is kept arbitrary, then univalence cannot be applied because equality is defined only for elements of the same type.
\begin{proof}
Because a sum indexed by a subsingleton is equivalent to any of its
summands, and similarly a product indexed by a subsingleton is equivalent to
any of its factors, and because a map is an embedding precisely when
its fibers are all subsingletons.
\end{proof}
\noindent We record this corollary:
\begin{lemma} \label{ref:16:1}
  The universe \m{\U \sqcup \V} is algebraically \m{\U,\V}-injective, in at least two ways.
\end{lemma}
\noindent And in particular, e.g.\ \m{\U} is \m{\U,\U}-injective, but of course
\m{\U} doesn't live in \m{\U} and doesn't even have a copy in
\m{\U}. For the following, we say that \m{y : Y} is not in the image
of \m{j} to mean that \m{j \, x \ne y} for all \m{x:X}.
\begin{proposition}
        For \m{y:Y} not in the image of \m{j}, we have
        \m{(f \edown j) \, y \simeq \Zero} and
        \m{(f \eup j) \, y \simeq \One}.
  \end{proposition}
\noindent With excluded middle, this would give that the two extensions have
the same sum and product as the non-extended map, respectively, but
excluded middle is not needed, as it is not hard to see:
\begin{remark} We have canonical equivalences
\m{\Sigma f \simeq \Sigma (f \edown j)} and
\m{\Pi f \simeq \Pi (f \eup j)}.
\end{remark}
Notice that the functions \m{f}, \m{f \edown j} and \m{f \eup j},
being universe valued, are type families, and hence the notations
\m{\Sigma f}, \m{\Sigma(f \edown j)}, \m{\Pi f} and \m{\Pi(f \eup j)}
are just particular cases of the notations for the sum and product of
a type family.

The two extensions are left and right Kan extensions in the following
sense, without the need to assume that \m{j} is an embedding. First, a
map \m{f:X \to \U}, when \m{X} is viewed as an \m{\infty}-groupoid and
hence an \m{\infty}-category, and when \m{\U} is viewed as the
\m{\infty}-generalization of the category of sets, can be considered
as a sort of \m{\infty}-presheaf, because its functoriality is
automatic: If we define
\M{f [ p ] \eqdef \transport f p}
of type \m{f\, x \to f\, y} for \m{p : \Id \, x \, y}, then for
\m{q : \Id \, y \, z} we have
\M{ f [
  \refl_x ] = \id_{f \, x}, \qquad\qquad f [p \pcomp q] = f [q] \comp f
  [p].
}
Then we can consider the type of transformations between such
\m{\infty}-presheaves \m{f : X \to \W} and \m{f' : X \to \W'} defined by
\M{
  f \wps f' \eqdef \Pi (x : X), f \, x \to  f' x,
}
which are automatically natural in the sense that for all \m{\tau: f \wps f'} and \m{p : \Id \, x \, y},
\M{
  \tau_y \comp f [ p ] = f' [p] \comp \tau_x.
}

It is easy to check that we have the following canonical transformations:
\begin{remark}
  \m{f \edown j \wps f \eup j} if \m{j} is an embedding.
\end{remark}
It is also easy to see that, without assuming \m{j} to be an embedding,
  \begin{enumerate}
  \item \m{f \wps f \edown j \comp j},
  \item \m{f \eup j \comp j \wps f}.
  \end{enumerate}
  These are particular cases of the following constructions, which are
  evident and canonical, even if they may be a bit
  laborious:
\begin{remark} For any \m{g : Y \to \T}, we have canonical equivalences
    \begin{enumerate}
    \item \m{(f \edown j \wps g) \simeq (f \wps g \comp j),} \quad i.e.\ \m{f \edown j} is a left Kan extension,
    \item \m{(g \wps f \eup j) \simeq (g \comp j \wps f),} \quad i.e.\ \m{f \eup j} is a right Kan extension.
    \end{enumerate}
\end{remark}

We also have that the left and right Kan extension operators along an
embedding are themselves embeddings, as we now show.
\begin{theorem}
For any types \m{X,Y:\U} and any embedding \m{j : X \to Y}, left Kan extension along \m{j} is an embedding of the function type \m{X \to \U} into the function type \m{Y \to \U}.
\end{theorem}
\begin{proof}
  Define \m{s : (X \to \U) \to (Y \to \U)} and \m{r : (Y \to \U) \to (X \to \U)} by
  \M{
    \begin{array}{lll}
      s \, f & \eqdef & f \edown j, \\
      r \,  g & \eqdef & g \comp j.
    \end{array}
  }
  By function extensionality, we have that \m{r (s \, f) = f}, because
  \m{s} is a pointwise-extension operator as \m{j} is an embedding,
  and by construction we have that \m{s (r \, g) = (g \comp j) \edown
    j}. Now define \m{\kappa : \Pi (g : Y \to \U), s(r \,g) \wps
    g} by
  \M{
    \kappa \, g \, y \, ((x , p) , C) \eqdef \transport \, g \, p \, C
  }
  for all \m{g : Y \to \U}, \m{y : Y}, \m{x : X}, \m{p : j \, x = y} and \m{C : g(j \, x)}, so that
  \m{\transport \, g \, p \, C} has type \m{g \, y },
  and consider the type
  \M{
    M \eqdef \Sigma (g : Y \to \U)\,\Pi(y:Y), \text{the map \m{\kappa \, g \, y : s (r \, g) \, y \to g \, y} is an equivalence.}
  }
  Because the notion of being an equivalence is a proposition and
  because products of propositions are propositions, the first projection
  \M{\fst : M \to (Y \to \U)} is an embedding.  To complete the proof,
  we show that there is an equivalence \m{\phi : (X \to \U) \to M}
  whose composition with this projection is \m{s}, so that \m{s},
  being the composition of two embeddings, is itself an embedding.  We
  construct \m{\phi} and its inverse \m{\gamma} by
  \M{
    \begin{array}{lll}
      \phi \, f & \eqdef & (s f , \varepsilon \, f), \\
      \gamma \, (g , e) & \eqdef & r \, g,
    \end{array}
  }
  where \m{\varepsilon \, f} is a proof that the map \m{\kappa \, (s
    f) \, y} is an equivalence for every \m{y : Y}, to be constructed
  shortly.  Before we know this construction, we can see that \m{\gamma
    (\phi \, f) = r (s \, f) = f} so that \m{\gamma \comp \phi \sim
    \id}, and that \m{\phi (\gamma (g , e)) = (s(r g) , \varepsilon (r
    g))}.  To check that the pairs \m{(s(r g) , \varepsilon (r g))}
  and \m{(g , e)} are equal and hence \m{\phi \comp \gamma \sim \id},
  it suffices to check the equality of the first components, because
  the second components live in subsingleton types.  But \m{e \, y}
  says that \m{s (r \, g) \, y \simeq g \, y} for any \m{y:Y}, and
  hence by univalence and function extensionality, \m{s (r \, g) = g}.
  Thus the functions \m{\phi} and \m{\gamma} are mutually
  inverse. Now, \m{\fst \comp \phi = s} definitionally using the
  $\eta$-rule for \m{\Pi}, so that indeed \m{s} is the composition of
  two embeddings, as we wanted to show.

  It remains to show that the map \m{\kappa \, (s f) \, y : s (f \, y) \to s(r(s \,
    f)) \, y} is indeed an equivalence. The domain and
  codomain of this function amount, by construction, to respectively
  \M{
    \begin{array}{lll}
      A & \eqdef & \Sigma (t : j^{-1}(y)), \Sigma (w : j^{-1}(j (\fst t))), f (\fst w)\\
      B & \eqdef & \Sigma (w : j^{-1}(y)), f(\fst w).
    \end{array}
  }
  We construct an inverse \m{\delta : B \to A} by
  \M{
    \delta \, ((x , p),C) \eqdef ((x , p) , (x , \refl_{j \, x}) , C).
  }
  It is routine to check that the functions \m{\kappa \, (s f) \, y}
  and \m{\delta} are mutually inverse, which concludes the proof.
\end{proof}

The proof of the theorem below follows the same pattern as the
previous one with some portions ``dualized'' in some sense, and so we
are slightly more economic with its formulation this time.

\begin{theorem}
For any types \m{X,Y:\U} and any embedding \m{j : X \to Y}, the right Kan extension operation along \m{j} is an embedding of the function type \m{X \to \U} into the function type \m{Y \to \U}.
\end{theorem}
\begin{proof}
  Define \m{s : (X \to \U) \to (Y \to \U)} and \m{r : (Y \to \U) \to (X \to \U)} by
  \M{
    \begin{array}{lll}
      s \, f & \eqdef & f \eup j, \\
      r \,  g & \eqdef & g \comp j.
    \end{array}
  }
  By function extensionality, we have that \m{r (s \, f) = f},
  and, by construction, \m{s (r \, g) = (g \comp j) \eup
    j}. Now define \m{\kappa : \Pi (g : Y \to \U), g \wps s(r \,g)
    } by
  \M{
    \kappa \, g \, y \, C (x , p) \eqdef \transport \, g \, p^{-1} \, C
  }
  for all \m{g : Y \to \U}, \m{y : Y}, \m{C : g \, y}, \m{x : X}, \m{p : j \, x = y}, so that
  \m{\transport \, g \, p^{-1} \, C} has type \m{g (j \, x) },
  and consider the type
  \M{
    M \eqdef \Sigma (g : Y \to \U)\,\Pi(y:Y), \text{the map \m{\kappa \, g \, y : g \, y \to s (r \, g) \,y} is an equivalence.}
  }
  Then the first projection \m{\fst : M \to (Y \to \U)} is an embedding.  To complete the proof,
  we show that there is an equivalence \m{\phi : (X \to \U) \to M}
  whose composition with this projection is \m{s}, so that it follows that \m{s}
  is an embedding.  We
  construct \m{\phi} and its inverse \m{\gamma} by
  \M{
    \begin{array}{lll}
      \phi \, f & \eqdef & (s f , \varepsilon \, f), \\
      \gamma \, (g , e) & \eqdef & r \, g,
    \end{array}
  }
  where \m{\varepsilon \, f} is a proof that the map \m{\kappa \, (s
    f) \, y} is an equivalence for every \m{y : Y}, so that \m{\phi}
  and \m{\gamma} are mutually inverse by the argument of the previous
  proof.

  To prove that the map \m{\kappa \, (s f) \, y : s(r(s \,
    f)) \, y \to s (f \, y)} is an equivalence, notice that its domain and
  codomain amount, by construction, to respectively
  \M{
    \begin{array}{lll}
      A & \eqdef & \Pi (w : j^{-1}(y)), f(\fst w), \\
      B & \eqdef & \Pi (t : j^{-1}(y)), \Pi (w : j^{-1}(j (\fst t))), f (\fst w).
    \end{array}
  }
  We construct an inverse \m{\delta : B \to A} by
  \M{
    \delta \, C \, (x , p) \eqdef C (x , p) (x , \refl_{j \, x}).
  }
  It is routine to check that the functions \m{\kappa \, (s f) \, y}
  and \m{\delta} are mutually inverse, which concludes the proof.
\end{proof}

The left and right Kan extensions trivially satisfy \m{f \edown \id \sim f}
and \m{f \eup \id \sim f} because the identity map is an embedding, by
the extension property, and so are contravariantly functorial in view
of the following.
\begin{remark} \label{iterated}
  For types \m{X : \U}, \m{Y : \V} and \m{Z : \W}, and functions \m{j : X \to Y}, \m{k : Y \to Z} and \m{f : X \to \U \sqcup \V \sqcup \W}, we have canonical identifications
  \M{
    \begin{array}{lll}
      f \edown (k \comp j)  & \sim & (f \edown j) \edown k, \\
    f \eup (k \comp j)  & \sim & (f \eup j) \eup k.
    \end{array}
  }
  %
\end{remark}
\begin{proof}
This is a direct consequence of the canonical equivalences
  \M{
    \begin{array}{lll}
      (\Sigma (t : \Sigma B) , C \, t) \simeq (\Sigma (a : A)\,
    \Sigma (b : B \, a), C(a,b)) \\
      (\Pi (t : \Sigma B) , C \, t) \simeq (\Pi (a : A)\,
    \Pi (b : B \, a), C(a,b))
    \end{array}
  }
  for arbitrary universes \m{\U,\V,\W} and \m{A:\U}, \m{B: A \to \V}, and \m{C : \Sigma \, B \to \W}.
\end{proof}

The above and the following are applied in work on
compact ordinals (reported in our repository~\cite{TypeTopology}).
\begin{remark}
  For types \m{X : \U} and \m{Y : \V}, and functions \m{j : X \to Y}, \m{f : X \to \W}
  and \m{f' : X \to \W'}, if the type \m{f \, x} is a retract of \m{f'
    \, x} for any
  \m{x:X}, then the type \m{(f \eup j) \, y} is a
  retract of \m{(f' \eup j) \, y} for any \m{y : Y}.
\end{remark}

\noindent The construction is routine, and presumably can be
performed for left Kan extensions too, but we haven't paused to check
this.

\section{Constructions with algebraically injective types}

Algebraic injectives are closed under retracts:
\begin{lemma}
  If a type \m{D} in a universe \m{\W} is algebraically
  \m{\U,\V}-injective, then so is any retract \m{D' : \W'} of \m{D} in
  any universe \m{\W'}.
\end{lemma}
\noindent In particular, any type equivalent to an algebraically injective type
is itself algebraically injective, without the need to invoke univalence.
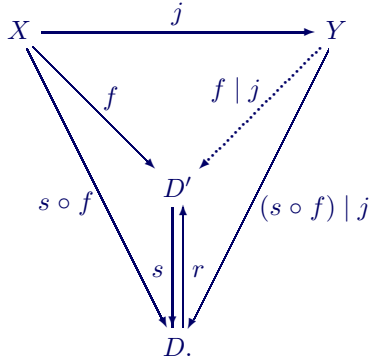
\begin{proof}
\M{\begin{diagram}[p=0.4em]
X &	 	& \rTo^j  & 		   & Y 	\\
  & \rdTo^f\rdTo(2,4)_{s \comp f} 	&   	       & \ldEto^{f \mid j}\ldTo(2,4)_{(s \comp f) \mid j}       & 	\\
  & 		& D'	       & 		   &    \\
  & 		& \dTo^s \uTo_r & 	   &    \\
  &		& D.
\end{diagram}}

\noindent For a given section-retraction pair \m{(s,r)}, the construction of the
extension operator for \m{D'} from that of \m{D} is given by \m{f \mid
  j \eqdef r \comp ((s \comp f) \mid j)}.
\end{proof}

\begin{lemma}
  The product of any family \m{D_a} of algebraically
  \m{\U,\V}-injective types in a universe \m{\W}, with indices \m{a} in a type
  \m{A} of any universe \m{\T}, is itself algebraically \m{\U,\V}-injective.
\end{lemma}
\noindent In particular, if a type \m{D} in a universe \m{\W} is algebraically
\m{\U,\V}-injective, then so is any exponential power \m{A \to D : \T \sqcup \W} for
any type \m{A} in any universe \m{\T}.
\begin{proof}
We construct the extension operator \m{(-)\mid(-)} of the product
\m{\Pi D : \T \sqcup \W} in a pointwise fashion from the extension
operators \m{(-)\mid_a(-)} of the algebraically injective types
\m{D_a}: For \m{f : X \to \Pi D}, we let \m{f \mid
  j : Y \to \Pi D} be
\M{
  (f \mid j) \, y \eqdef a \mapsto ((x \mapsto f \, x \, a) \mid_a j) \, y.
}
\end{proof}

\begin{lemma}
  Every algebraically \m{\U,\V}-injective type \m{D:\W} is a retract
  of any type \m{Y:\V} in which it is embedded into.
\end{lemma}
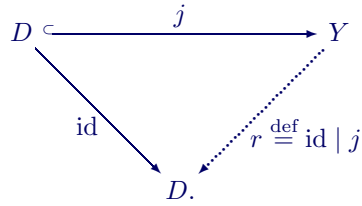
\begin{proof}
\M{\begin{diagram}
  D & & \rInto^j & & Y  \\
  & \rdTo_\id &  & \ldEto_{r \eqdef \id \mid j} & \\
  & & D. & &
\end{diagram}}

\noindent
We just extend the identity function along the embedding to get the desired retraction~\m{r}.
\end{proof}

The following is a sort of \m{\infty}-Yoneda embedding:
\begin{lemma}
The identity type former \m{\Id_X} of any type \m{X:\U} is an embedding of the type~\m{X} into the type~\m{X \to \U}.
\end{lemma}
\begin{proof}
  To show that the \m{\Id}-fiber of a given \m{A : X \to \U} is a
  subsingleton, it suffices to show that if is pointed then it is
  a singleton.  So let \m{(x,p):\Sigma (x : X), \Id x = A} be a point
  of the fiber. Applying \m{\Sigma}, seen as a map of type \m{(X \to
    \U) \to \U}, to the identification~\m{p : \Id \, x = A}, we get an
  identification
  \M {
  \ap \, \Sigma \, p : \Sigma (\Id x) = \Sigma A,
}
and hence, being equal to the singleton type \m{\Sigma (\Id x)},
the type \m{\Sigma A} is itself a singleton.  Hence we have
  \M{\begin{array}{llll}
       A \, x & \simeq & \Id x \wps A & \text{By the Yoneda Lemma~\cite{rijke:msc},} \\
           & = & \Pi (y : X), \Id \, x \, y \to A \, y & \text{by definition of \m{\wps},} \\
           & \simeq & \Pi (y : X), \Id \, x \, y \simeq A \, y & \text{because \m{\Sigma A} is a singleton (Yoneda corollary),} \\
           & \simeq & \Pi (y : X), \Id \, x \, y = A \, y & \text{by univalence,} \\
           & \simeq & \Id \, x = A & \text{by function extensionality.}
     \end{array}
  }
  So by a second application of univalence we get \m{A \, x = (\Id \, x
    = A)}. Hence, applying \m{\Sigma} on both sides, we get
  \m{\Sigma A = (\Sigma (x : X), \Id \, x = A)}.  Therefore, because
  the type \m{\Sigma A} is a singleton, so is the fiber \m{\Sigma (x
    : X), \Id \, x = A} of~\m{A}.
\end{proof}
\begin{lemma} \label{ref:16:3}
  If a type \m{D} in a universe \m{\U} is algebraically \m{\U,\U^+}-injective, then \m{D} is a retract of the exponential power \m{D \to \U} of \m{\U}.
\end{lemma}
\begin{proof}
\M{\begin{diagram}
  D & & \rInto^\Id & & (D \to \U)  \\
  & \rdTo_\id &  & \ldEto_{r \eqdef \id \mid \Id} & \\
  & & D. & &
\end{diagram}}

\noindent This is obtained by combining the previous two
constructions, using the fact that \m{D \to \U} lives in the successor
universe \m{\U^+}.
\end{proof}

\section{Algebraic flabbiness and resizing constructions}

We now discuss resizing constructions that don't assume resizing
axioms.  The above results, when combined together in the obvious way,
almost give directly that the algebraically injective types are
precisely the retracts of exponential powers of universes, but there
is a universe mismatch.  Keeping track of the universes to avoid the
mismatch, what we get instead is a resizing construction without the
need for resizing axioms:
\begin{lemma}
  Algebraically \m{\U,\U^+}-injective types \m{D:\U} are
  algebraically \m{\U,\U}-injective too.
\end{lemma}
\begin{proof}
By the above constructions, we
first get that \m{D}, being algebraically \m{\U,\U^+}-injective, is a
retract of \m{D \to \U}. But then \m{\U} is algebraically
\m{\U,\U}-injective, and, being a power of \m{\U}, so is \m{D \to \U}.
Finally, being a retract of \m{D \to \U}, we have that \m{D} is
algebraically \m{\U,\U}-injective.
\end{proof}
This is resizing down and so is not surprising.
Of course, such a construction can be performed directly by considering
an embedding \m{\U \to \U^+}, but the idea is to generalize it to
obtain further resizing-for-free constructions, and, later,
resizing-for-a-price constructions.  We achieve this by considering a
notion of flabbiness as data, rather than as property as in the
1-topos literature (see e.g.\
Blechschmidt~\cite{Blechschmidt:2018}). The notion of flabbiness
considered in topos theory is defined with truncated \m{\Sigma}, that
is, the existential quantifier \m{\exists} with values in the
subobject classifier \m{\Omega}. We refer to the notion defined with
untruncated \m{\Sigma} as algebraic flabbiness.

\begin{definition}
  We say that a type \m{D : \W} is \df{algebraically \m{\U}-flabby} if
\M{
  \Pi (P : \U), \text{if \m{P} is a subsingleton then \m{\Pi(f : P \to D)\, \Sigma (d : D)\, \Pi(p : P), d = f \, p}.}
}
\end{definition}
\noindent This terminology is more than a mere analogy with
algebraic injectivity: notice that flabbiness and algebraic flabbiness
amount to simply injectivity and algebraic injectivity with respect to
the class of embeddings \m{P \to \One} with \m{P} ranging over
subsingletons:
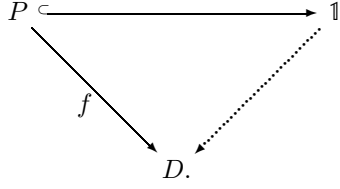
\begin{diagram}
  P & & \rInto & & \One  \\
  & \rdTo_f &  & \ldEto & \\
  & & D. & &
\end{diagram}
Notice also that an algebraically flabby type \m{D} is pointed, by
considering the case when \m{f} is the unique map \m{\Zero \to D}.

\begin{lemma} \label{for27:1}
  If a type \m{D} in the universe \m{\W} is algebraically
  \m{\U,\V}-injective, then it is algebraically \m{\U}-flabby.
\end{lemma}
\begin{proof}
  Given a subsingleton \m{P:\U} and a map \m{f : P \to D}, we can take its extension \m{f \mid \operatorname{!}: \One \to D} along the unique map \m{!:P \to \One}, because it is an embedding, and then we let \m{d \eqdef (f \mid \operatorname{!})\, \one}, and the extension property gives \m{d = f \, p} for any \m{p:P}.
\end{proof}
The interesting thing about this is that the
universe~\m{\V} is forgotten, and then we can put any other universe
below \m{\U} back, as follows.

\begin{lemma} \label{for27:2}
  If a type \m{D} in the universe \m{\W} is algebraically \m{\U \sqcup
    \V}-flabby, then it is also algebraically \m{\U,\V}-injective.
\end{lemma}
\begin{proof}
  Given an embedding \m{j : X \to Y} of types \m{X:\U} and \m{\V}, a map \m{f : X \to D} and a
  point \m{y:Y}, in order to construct \m{(f \mid j) \, y} we consider
  the map \m{f_y : j^{-1}(y) \to D} defined by \m{(x,p) \mapsto
    f\,x}. Because the fiber \m{j^{-1}(y) : \U \sqcup \V} is a subsingleton as \m{j} is an
  embedding, we can apply algebraic flabbiness to get \m{d_y : D} with
  \m{d_y = f_y (x,p)} for all \m{(x,p):j^{-1}(y)}. By the construction of
  \m{f_y} and the definition of fiber, this amounts to saying that for
  any \m{x : X} and \m{p : j \, x = y}, we have \m{d_y = f \,
    x}. Therefore we can take
\M{(f \mid j) \, y \eqdef d_y,}
because we then have
\M{(f \mid j) (j \, x) = d_{j \, x} = f_{j \, x} (x , \refl_{j \, x}) = f \, x}
for any \m{x:X}, as required.
\end{proof}
\noindent We then get the following resizing construction by composing the above
two conversions between algebraic flabbiness and injectivity:
\begin{lemma}
  If a type \m{D} in the universe \m{\W} is algebraically \m{(\U \sqcup
    \T),\V}-injective, then it is also algebraically \m{\U,\T}-injective.
\end{lemma}
\noindent In particular, algebraic \m{\U,\V}-injectivity gives
algebraic \m{\U,\U}- and \m{\U_0,\U}-injectivity.  So this is no
longer necessarily resizing down, by taking \m{\V} to be
e.g.\ the first universe~\m{\U_0}.

\section{Injectivity of subuniverses}

We now apply algebraic flabbiness to show that any subuniverse closed
under subsingletons and under sums, or alternatively under products,
is also algebraically injective.
\begin{definition}
  By a \df{subuniverse} of \m{\U} we mean a projection \m{\Sigma \, A
    \to \U} with \m{A : \U \to \T} subsingleton-valued and the
  universe \m{\T} arbitrary. By a customary abuse of language, we also
  sometimes refer to the domain of the projection as the
  subuniverse. Closure under subsingletons means that \m{A\,P} holds
  for any subsingleton \m{P:\U}. Closure under sums amounts to
  saying that if \m{X:\U} satisfies \m{A} and every \m{Y \, x}
  satisfies \m{A} for a family \m{Y : X \to \U}, then so does
  \m{\Sigma \, Y}. Closure under products is defined in the same way
  with \m{\Pi} in place of \m{\Sigma}.
\end{definition}
\noindent Notice that \m{A} being subsingleton-valued is
precisely what is needed for the projection to be an embedding, and
that all embeddings are of this form up to equivalence (more
precisely, every embedding of any two types is the composition of an
equivalence into a sum type followed by the first projection).

\begin{lemma}
  Any subuniverse of \m{\U} which is closed
  under subsingletons and sums, or alternatively under subsingletons and
  products, is algebraically \m{\U}-flabby and hence
  algebraically \m{\U,\U}-injective.
\end{lemma}
\begin{proof}
  Let \m{\Sigma\,A} be a subuniverse of \m{\U}, let \m{P:\U} be a
  subsingleton and \m{f : P \to \Sigma \, A} be given. Then define
  \begin{quote}
  (1)~\m{ X \eqdef \Sigma (\fst \comp f)} \qquad or \qquad (2)~\m{X \eqdef \Pi (\fst \comp f)}
  \end{quote}
  according to whether we have closure under sums or products. Because
  \m{P}, being a subsingleton satisfies \m{A} and because the values
  of the map \m{\fst \comp f : P \to \U} satisfy \m{A} by definition
  of subuniverse, we have \m{a : A\, X} by the sum or product closure
  property, and \m{d \eqdef (X,a)} has type \m{\Sigma \,A}. To
  conclude the proof, we need to show that \m{d = f\,p} for any
  \m{p:P}. Because the second component \m{a} lives in a subsingleton
  by definition of subuniverse, it suffices to show that the first
  components are equal, that is, that \m{X = \fst (f p)}. But this
  follows by univalence, because a sum indexed by a subsingleton is
  equivalent to any of summands, and a product indexed by a
  subsingleton is equivalent to any of its factors.
\end{proof}

We index \m{n}-types from \m{n=-2} as in the HoTT book, where the
\m{-2}-types are the singletons. We have the following as a corollary.
\begin{theorem}
  The subuniverse of \m{n}-types in a universe \m{\U} is algebraically
  \m{\U}-flabby, in at least two ways, and hence algebraically
  \m{\U,\U}-injective.
\end{theorem}
\begin{proof}
  We have a subuniverse because the notion of being an \m{n}-type is a
  proposition. For \m{n=-2}, the subuniverse of singletons is itself a
  singleton, and hence trivially injective.  For \m{n>-2}, the
  \m{n}-types are known to be closed under subsingletons and both sums
  and products.
\end{proof}

\noindent In particular:
\begin{enumerate}
\item The type \m{\Omega_\U} of subsingletons in a universe \m{\U} is
  algebraically \m{\U,\U}-injective.

  (Another way to see that \m{\Omega_\U} is algebraically injective is
  that it is a retract of the universe by propositional
  truncation. The same would be the case for \m{n}-types if we were
  assuming \m{n}-truncations, which we are not.)

\item Powersets, being exponential powers of \m{\Omega_\U}, are
  algebraically \m{\U,\U}-injective.
\end{enumerate}
An anonymous referee suggested the following additional examples: (i)
The subuniverse of subfinite types, i.e., subtypes of types for which
there is an uunspecified equivalence with \m{\operatorname{Fin}(n)}
for some~\m{n}. This subuniverse is closed under both \m{\Pi} and
\m{\Sigma}.  (ii) Reflective subuniverses, as they are closed under
\m{\Pi}. (iii) Any universe \m{\U} seen as a subuniverse of \m{\U
  \sqcup \V}.

\section{Algebraic flabbiness with resizing axioms}

Returning to size issues, we now apply algebraic flabbiness to show
that propositional resizing gives unrestricted algebraic injective
resizing.
\begin{definition} \label{resizing}
The propositional resizing principle, from \m{\U} to \m{\V}, that we
consider here says that every proposition in the universe \m{\U} has
an equivalent copy in the universe~\m{\V}. By propositional resizing without
qualification, we mean propositional resizing between any of the
universes involved in the discussion.
\end{definition}
This is consistent because
it is implied by excluded middle, but, as far as we are aware, there
is no known computational interpretation of this axiom. A model in
which excluded middle fails but propositional resizing holds is given
by Shulman~\cite{MR3340541}.

We begin with the following construction, which says that algebraic
flabbiness is universe independent in the presence of propositional
resizing:

\begin{lemma}
  If propositional resizing holds, then the algebraic \m{\V}-flabbiness
  of a type in any universe gives its algebraic \m{\U}-flabbiness.
\end{lemma}
\begin{proof}
  Let \m{D:\W} be a type in any universe \m{\W}, let \m{P : \U} be a
  proposition and \m{f : P \to D}. By resizing, we have an equivalence
  \m{\beta : Q \to P} for a suitable proposition \m{Q:\V}.  Then the
  algebraic \m{\V}-flabbiness of \m{D} gives a point \m{d:D} with \m{d
    = (f \comp \beta) \, q} for all \m{q : Q}, and hence with \m{d = f
    \, p} for all \m{p : P}, because we have \m{p=\beta \, q} for \m{q
    = \alpha \, p} where \m{\alpha} is a quasi-inverse of \m{\beta},
  which establishes the algebraic \m{\U}-flabbiness of~\m{D}.
\end{proof}

And from this it follows that algebraic injectivity is also universe
independent in the presence of propositional resizing: we convert
back-and-forth between algebraic injectivity and algebraic flabbiness.

\begin{lemma} \label{universe:independence}
  If propositional resizing holds, then for any type \m{D} in any universe
  \m{\W}, the algebraic \m{\U,\V}-injectivity of \m{D} gives its
  algebraic \m{\U',\V'}-injectivity.
\end{lemma}
\begin{proof}
  We first get the \m{\U}-flabbiness of \m{D} by~\ref{for27:1}, and then its
  \m{\U' \sqcup \V'}-flabbiness by~\ref{universe:independence}, and finally its algebraic
  \m{\U',\V'}-injectivity by~\ref{for27:2}.
\end{proof}

As an application of this and of the algebraic injectivity of
universes, we get that any universe is a retract of any larger
universe.  We remark that for types that are not sets, sections are
not automatically embeddings~\cite{MR3548859}. But we can choose the
retraction so that the section is an embedding in our situation.

\begin{lemma} \label{canonical}
  We have an embedding of any universe \m{\U} into any larger universe \m{\U \sqcup \V}.
\end{lemma}
\begin{proof}
  For example, we have the embedding given by \m{X \mapsto X +
    \Zero_\V}. We don't consider an argument that this is indeed an
  embedding to be entirely
  routine without a significant amount of experience in univalent
  mathematics, even if this may seem obvious. Nevertheless, it is certainly
  safe to leave it as a challenge to the reader, and a proof can be
  found in~\cite{injective:article} in case of doubt.
\end{proof}

\begin{theorem}
  If propositional resizing holds, then any universe \m{\U} is a
  retract of any larger universe \m{\U \sqcup \V} with a section that
  is an embedding.
\end{theorem}
\begin{proof}
  The universe \m{\U} is algebraically \m{\U,\U}-injective
  by~\ref{ref:16:1}, and hence it is algebraically \m{\U^+,(\U \sqcup
    \V)^+}-injective by~\ref{universe:independence}, which has the
  right universe assignments to apply the construction~\ref{ref:16:3}
  that gives a retraction from an embedding of an injective type into
  a larger type, in this case the embedding of the universe \m{\U}
  into the larger universe \m{\U \sqcup \V} constructed
  in~\ref{canonical}.
\end{proof}

As mentioned above, we almost have that the algebraically injective
types are precisely the retracts of exponential powers of universes,
up to a universe mismatch. This mismatch is side-stepped by
propositional resizing. The following is one of the main results of
this paper:

\begin{theorem} \df{(First characterization of algebraic
    injectives.)}  If propositional resizing holds, then a type \m{D}
  in a universe \m{\U} is algebraically \m{\U,\U}-injective if and
  only if \m{D} is a retract of an exponential power of \m{\U} with
  exponent in \m{\U}.
\end{theorem}
\noindent We emphasize that this is a logical equivalence ``if and
only if'' rather than an \m{\infty}-groupoid equivalence
``\m{\simeq}''. More precisely, the theorem gives two constructions in
opposite directions. So this characterizes the types that \df{can} be
equipped with algebraic-injective structure.
\begin{proof}
  \m{(\Rightarrow)}: Because \m{D} is algebraically
  \m{\U,\U}-injective, it is algebraically \m{\U,\U^+}-injective by
  resizing, and hence it is a retract of \m{D \to \U} because it is
  embedded into it by the identity type former, by taking the
  extension of the identity function along this embedding.

  \m{(\Leftarrow)}: If \m{D} is a retract of \m{X \to \U} for some
  given \m{X:\U}, then, because \m{X \to \U}, being an exponential
  power of the algebraically \m{\U ,\U}-injective type \m{\U}, is
  algebraically \m{\U,\U}-injective, and hence so is \m{D} because it
  is a retract of this power.
\end{proof}

We also have that any algebraically injective \m{(n+1)}-type is a retract
of an exponential power of the universe of \m{n}-types. We establish something
more general first.

\begin{lemma}
  Under propositional resizing, for any subuniverse \m{\Sigma \, A} of
  a universe \m{\U} closed under subsingletons, we have that any
  algebraically \m{\U,\U}-injective type \m{X:\U} whose identity types
  \m{x=_X x'} all satisfy the property \m{A} is a retract of
  the type \m{X \to \Sigma \, A}.
\end{lemma}
\begin{proof}
  Because the first projection \m{j : \Sigma \, A \to \U} is an
  embedding by the assumption, so is the map \m{k \eqdef j \comp (-) :
    (X \to \Sigma A) \to (X \to \U)} by a general property of
  embeddings. Now consider the map \m{l : X \to (X \to \Sigma \, A)}
  defined by \m{x \mapsto (x' \mapsto (x=x', p \, x \, x'))}, where
  \m{p \, x \, x' : A(x=x')} is given by the assumption. We have that
  \m{k \comp l = \Id_X} by construction. Hence \m{l} is an embedding
  because \m{l} and \m{\Id_X} are, where we are using the general fact
  that if \m{g \comp f} and \m{g} are embeddings then so is the
  factor~\m{f}.  But \m{X}, being algebraically \m{\U,\U}-injective by
  assumption, is algebraically \m{\U,(\U^+ \sqcup \T)}-injective by
  resizing, and hence so is the exponential power \m{X \to \Sigma \,
    A}, and therefore we get the desired retraction by extending its
  identity map along~\m{l}.
\end{proof}

Using this, we get the following as an immediate consequence.

\begin{theorem} \df{(Characterization of algebraic injective \m{(n+1)}-types.)}  If propositional resizing holds, then an
  \m{(n+1)}-type \m{D} in \m{\U} is algebraically \m{\U,\U}-injective
  if and only if \m{D} is a retract of an exponential power of the
  universe of \m{n}-types in \m{\U}, with exponent in \m{\U}.
\end{theorem}

\begin{corollary}
  The algebraically injective sets in \m{\U} are the retracts of
  powersets of (arbitrary) types in \m{\U}, assuming propositional resizing.
\end{corollary}
\noindent Notice that the powerset of any type is a set, because
\m{\Omega_\U} is a set and because sets (and more generally
\m{n}-types) form an exponential ideal.

\section{Injectivity in terms of algebraic injectivity in the absence of resizing}

We now compare injectivity with algebraic injectivity. The following observation follows from the fact that retractions are surjections:
\begin{lemma}
  If a type \m{D} in a universe \m{\W} is algebraically
  \m{\U,\V}-injective, then it is \m{\U,\V}-injective
\end{lemma}
\noindent The following observation follows from the fact that propositions are
closed under products.
\begin{lemma}
  Injectivity is a proposition.
\end{lemma}
\noindent But of course algebraic injectivity is not. From this we immediately
get the following by the universal property of propositional
truncation:

\begin{lemma}
  For any type \m{D} in a universe \m{\W}, the truncation of the
  algebraic \m{\U,\V}-injectivity of \m{D} gives its
  \m{\U,\V}-injectivity.
\end{lemma}

In order to relate injectivity to the propositional truncation of
algebraic injectivity in the other direction, we first establish some
facts about injectivity that we already proved for algebraic
injectivity. These facts cannot be obtained by reduction (in
particular products of injectives are not necessarily injective, in
the absence of choice, but exponential powers are).

\begin{lemma} \label{embedding-||retract||}
  Any \m{\W,\V}-injective type \m{D} in a universe \m{\W} is a retract
  of any type in \m{\V} it is embedded into, in an unspecified way.
\end{lemma}
\begin{proof}
  Given \m{Y:\V} with an embedding \m{j : D \to Y}, by the
  \m{\W,\V}-injectivity of \m{D} there is an \df{unspecified} \m{r : Y
    \to D} with \m{r \comp j \sim \id}. Now, if there is a
  \df{specified} \m{r : Y \to D} with \m{r \comp j \sim \id} then
  there is a specified retraction. Therefore, by the functoriality of
  propositional truncation on objects applied to the previous
  statement, there is an unspecified retraction.
\end{proof}

\begin{lemma}
  If a type \m{D' : \U'} is a retract of a type \m{D : \U} then the \m{\W,\T}-injectivity of
  \m{D} implies that of \m{D'}.
\end{lemma}
\begin{proof}
  Let \m{r : D \to D'} and \m{s : D' \to D} be the given section
  retraction pair, and, to show that \m{D'} is \m{\W,\T}-injective,
  let an embedding \m{j : X \to Y} and a function \m{f : X \to D'} be
  given. By the injectivity of \m{D}, we have some unspecified
  extension \m{f' : Y \to D} of \m{s \comp f : X \to D}.  If such a
  designated extension is given, then we get the designated extension
  \m{r \comp f'} of \m{f}. By the functoriality of propositional
  truncation on objects and the previous two statements, we get the
  required, unspecified extension.
\end{proof}

The universe assignments in the following are probably not very
friendly, but we are aiming for maximum generality.
\begin{lemma}
  If a type \m{D : \W} is \m{(\U \sqcup \T),(\V \sqcup \T)}-injective,
  then the exponential power \m{A \to D} is \m{\U,\V}-injective for any \m{A:\T}.
\end{lemma}
\begin{proof}
  For a given embedding \m{j : X \to Y} and a given map \m{f : X \to
    (A \to D)}, take the exponential transpose \m{g : X \times A \to
    D} of \m{f}, then extend it along the embedding \m{j \times \id :
    X \times A \to Y \times A} to get \m{g' : Y \times A \to D} and
  then back-transpose to get \m{f' : Y \to (A \to D)}, and check that
  this construction of \m{f'} does give an extension of \m{f} along
  \m{j}. For this, we need to know that if \m{j} is an embedding then
  so is \m{j \times \id}, but this is not hard to check. The result
  then follows by the functoriality-on-objects of the propositional
  truncation.
\end{proof}

\begin{lemma}
  If a type \m{D:\U} is \m{\U,\U^+} injective, then it is a retract of
  \m{D \to \U} in an unspecified way.
\end{lemma}
\begin{proof}
  This is an immediate consequence of~\ref{embedding-||retract||} and the fact that the
  identity type former \m{\Id_X : X \to (X \to \U)} is an embedding.
\end{proof}

With this we get an almost converse to the fact that truncated
algebraic injectivity implies injectivity: the universe levels are
different in the converse:

\begin{lemma}
  If a type \m{D:\U} is \m{\U,\U^+}-injective, then it is algebraically \m{\U,\U^+}-injective in an unspecified way.
\end{lemma}

So, in summary, regarding the relationship between injectivity and
truncated algebraic injectivity, so far we know that
\begin{quote}
  if \m{D} is algebraically \m{\U,\V}-injective in an unspecified way
  then it is \m{\U,\V}-injective,
\end{quote}
and, not quite conversely,
\begin{quote}
  if \m{D} is \m{\U,\U^+}-injective then it is algebraically \m{\U,\U}-injective in an unspecified way.
\end{quote}
Therefore, using propositional resizing, we get the following
characterization of a particular case of injectivity in terms of
algebraic injectivity.

\begin{proposition} \label{worse} \df{(Injectivity in terms of algebraic injectivity.)}
  If propositional resizing holds, then a type \m{D : \U} is \m{\U,\U^+}-injective if and only if it is algebraically \m{\U,\U^+}-injective in an unspecified way.
\end{proposition}
\noindent We would like to do better than this. For that purpose, we consider
the partial-map classifier in conjunction with flabbiness and resizing.

\section{Algebraic flabbiness via the partial-map classifier}

We begin with a generalization~\cite{MR3695545} of a familiar construction
in \m{1}-topos theory~\cite{MR1173017}.
\begin{definition}
The lifting \m{\Lift_{\T} \, X : \T^+ \sqcup \U} of a type \m{X:\U}
with respect to a universe \m{\T} 
is defined by
\M{ \Lift_{\T}\, X \eqdef \Sigma (P : \T), (P \to X) \times
  \text{\m{P} is a subsingleton}.  }
\end{definition}

When the universes \m{\T} and \m{\U} are the same and the last
component of the triple is omitted, we have the familiar
canonical correspondence
\M{
  (X \to \T) \simeq (\Sigma (P : \T), P \to X)
}
that maps \m{A : X \to \T} to \m{P \eqdef \Sigma \, A} and the
projection \m{\Sigma \, A \to X}.  If the universe~\m{\U} is not
necessarily the same as \m{\T}, then the equivalence becomes
\M{
  (\Sigma (A : X \to \T \sqcup \U), \Sigma(T : \T), T \simeq \Sigma \, A) \simeq (\Sigma (P : \T), P \to X).
}
This says that although the total space \m{\Sigma \, A} doesn't
live in the universe \m{\T}, it must have a copy in \m{\T}.

What the third component of the triple does is to restrict the above
equivalences to the subtype of those \m{A} whose total spaces \m{\Sigma
  \, A} are subsingletons. If we define the type of partial maps by
\M{(X \partialto Y) \eqdef \Sigma (A : \T), (A \emb X) \times (A \to Y),
}
where \m{A \emb X} is the type of embeddings, then for any \m{X,Y :
  \T}, we have an equivalence
\M{
  (X \partialto Y)
  \simeq (X \to \Lift_{\T} \, Y),
}
so that \m{\Lift_{\T}} is the partial-map classifier for the universe
\m{\T}.  When the universe~\m{\U} is not necessarily the same
as~\m{\T}, the lifting classifies partial maps in~\m{\U} whose
embeddings have fibers with copies in~\m{\T}.

This is a sort of an \m{\infty}-monad ``across
universes''~\cite{TypeTopology}, and modulo providing coherence data,
which we haven't done at the time of writing, but which is not needed
for our purposes. We could call this a ``wild monad'', but we will
refer to it as simply a monad with this warning.

In order to discuss the lifting in more detail, we first characterize
its equality types. We denote the projections from \m{\Lift_{\T}
  \, X} by
\M{
  \begin{array}{llll}
    \delta (P , \phi , i) & \eqdef & P & \text{(domain of definition),} \\
    \upsilon (P , \phi , i) & \eqdef & \phi & \text{(value function),} \\
    \sigma(P , \phi , i) & \eqdef & i & \text{(subsingleton-hood of the domain of definition).}
  \end{array}
}
For \m{l , m : \Lift_{\T} \, X}, define
\M{
  (l \backsimeq m) \eqdef \Sigma (e : \delta \, l \simeq \delta \, m), \upsilon \, l = \upsilon \, m \comp e,
}
as indicated in the commuting triangle
\M{\begin{diagram}[p=0.4em]
  \delta l & & \rTo^e & & \delta m \\
  & \rdTo_{v l} & & \ldTo_{v m} & \\
   & & X & &
\end{diagram}}

\begin{lemma}
  The canonical transformation \m{(l = m) \to (l \backsimeq m)} that
  sends \m{\refl_l} to the identity equivalence paired with \m{\refl_{\upsilon \,
      l}} is an equivalence.
\end{lemma}

The unit \m{\eta : X
  \to \Lift_\T X} is given by \M{\eta_X \, x = (\One, (p \mapsto x), i)}
where \m{i} is a proof that \m{\One} is a proposition.
\begin{lemma}
  The unit \m{\eta_X : X\to\Lift_\T X} is an embedding.
\end{lemma}
\begin{proof}
  This is easily proved using the above characterization of equality.
\end{proof}
\begin{lemma}
  The unit satisfies the unit equations for a monad.
\end{lemma}
\begin{proof}
  Using the above characterization of equality, the left and right
  unit laws amount to the fact that the type \m{\One} is the left
  and right unit for the operation \m{(-)\times(-)} on types.
\end{proof}
\noindent Next, \m{\Lift_\T} is functorial by mapping a function
\m{f : X \to Y} to the function \m{\Lift_\T f : \Lift_\T X \to \Lift_\T Y}
defined by
\M{
\Lift_\T f (P , \phi , i) = (P , f \comp \phi , i).
}
This commutes with identities and composition definitionally.
%
%
We define the multiplication \m{\mu_X : \Lift_{\T} (\Lift_{\T}\, X) \to \Lift_{\T}\, X} by
\M{
  \begin{array}{lll}
  \delta (\mu (P , \phi , i)) & \eqdef & \Sigma (p : P), \delta (\phi \, p), \\
  \upsilon (\mu (P , \phi , i)) & \eqdef & (p , q) \mapsto \upsilon (\phi \, p) \, q , \\
  \sigma (\mu (P , \phi , i)) & \eqdef & \text{because subsingletons are closed under sums.} \\
  \end{array}
}
\begin{lemma}
  The multiplication satisfies the associativity equation for a monad.
\end{lemma}
\begin{proof}
  Using the above characterization of equality, we see that this
  amounts to the associativity of \m{\Sigma}, which says that for
  \m{P:\T}, \m{Q: X \to \T}, \m{R : \Sigma \, Q \to \T} we have
   \m{(\Sigma (t : \Sigma \, Q), R \, t) \simeq (\Sigma (p : P)\,
    \Sigma (q : Q \, p), R(p,q))}.
\end{proof}

\noindent The naturality conditions for the unit and multiplication are
even easier to check, and we omit the verification. We now turn to
algebras. We omit the direct verification of the following.

\begin{lemma} Let \m{X:\U} be any type.
  \begin{enumerate}
  \item A function \m{\alpha : \Lift_\T X \to X}, that is, a functor
    algebra, amounts to a family of functions \m{\bigsqcup_P : (P \to
      X) \to X} with \m{P : \T} ranging over subsingletons.

    \medskip We will write \m{\bigsqcup_P \phi} as \m{\bigsqcup_{p : P} \, \phi \, p}.
  \item The unit law for monad algebras amounts to, for any \m{x:X},
    \M{
       \bigsqcup_{p : \One} x = x,
    }
    which is equivalent to, for all subsingletons \m{P}, functions \m{\phi : P \to X} and points \m{p_0 : P},
    \M{
       \bigsqcup_{p : P} \phi \, p = \phi \, p_0.
    }

    \medskip Therefore a functor algebra satisfying the unit law
    amounts to the same thing as algebraic flabbiness data. In other
    words, the algebraically \m{\T}-flabby types are the algebras of
    the pointed functor \m{(\Lift_\T,\eta)}. In particular,
    monad algebras are algebraically flabby.
  \item The associativity law for monad algebras amounts to, for any subsingleton \m{P :
      \T} and family \m{Q : P \to \T} of subsingletons, and any \m{\phi : \Sigma \, Q \to X},
    \M{
      \bigsqcup_{t : \Sigma Q} \phi \, t = \bigsqcup_{p : P} \bigsqcup_{q : Q \, p} \phi (p ,q).
    }
  \end{enumerate}
\end{lemma}
\noindent So the associativity law for algebras plays no role in
flabbiness. But of course we can have algebraic flabbiness data that
is associative, such as not only the free algebra \m{\Lift_\T X}, but
also the following two examples that connect to the opening
development of this paper on the injectivity of universes, in
particular the construction~\ref{iterated}:
\begin{lemma} The universe \m{\T} is a monad algebra of
  \m{\Lift_\T} in at least two ways, with \m{\bigsqcup = \Sigma} and
  \m{\bigsqcup = \Pi}.
\end{lemma}

We now apply these ideas to injectivity.
\begin{lemma}
  Any algebraically \m{\T,\T^+}-injective type \m{D:\T} is a retract of \m{\Lift_\T D}.
\end{lemma}
\begin{proof}
  Because the unit is an embedding, and so we can extend the identity of~\m{D} along it.
\end{proof}

\begin{theorem} \df{(Second characterization of algebraic injectives.)}
  With propositional resizing, a type \m{D:\T} is algebraically
  \m{\T,\T}-injective if and only if it is a retract of a monad
  algebra of \m{\Lift_\T}.
\end{theorem}
\begin{proof}
  \m{(\Rightarrow)}: Because \m{D} is algebraically
  \m{\T,\T}-injective, it is algebraically \m{\T,\T^+}-injective by
  resizing, and hence it is a retract of \m{\Lift_\T D}.
  \m{(\Leftarrow)}: Algebraic injectivity is closed under retracts.
\end{proof}

\begin{definition} \label{omega:resizing}
Now, instead of propositional resizing, we consider the propositional
impredicativity of the universe \m{\U}, which says that the type
\m{\Omega_\U} of propositions in \m{\U}, which lives in the next
universe \m{\U^+}, has an equivalent copy in \m{\U}. We refer to this
kind of impredicativity as \m{\Omega}-resizing.
\end{definition}
It is not hard to see
that propositional resizing implies \m{\Omega}-resizing for all
universes other than the first one~\cite{TypeTopology}, and so all
the assumption of \m{\Omega}-resizing does is to account
for the first universe too.
\begin{lemma}
  Under \m{\Omega}-resizing, for any type \m{X:\T}, the type
  \m{\Lift_{\T} X : \T^+} has an equivalent copy in the universe \m{\T}.
\end{lemma}
\begin{proof}
  We can take \m{\Sigma (p : \Omega'), \fst(\rho \, p) \to X} where \m{\rho : \Omega' \to \Omega_\T} is the given equivalence.
\end{proof}

We apply this lifting machinery to get the following, which doesn't
mention lifting in its formulation.

\begin{theorem} \label{better}
  (Characterization of injectivity in terms of algebraic injectivity.)
  In the presence of \m{\Omega}-resizing, the
  \m{\T,\T}-injectivity of a type \m{D} in a universe \m{\T} is
  equivalent to the propositional truncation of
  its algebraic \m{\T,\T}-injectivity.
\end{theorem}
\begin{proof}
  We already know that the truncation of algebraic injectivity
  (trivially) gives injectivity.  For the other direction, let $L$ be
  a resized copy of \m{\Lift_\T D} in the universe \m{\T}.  Composing
  the unit with the equivalence given by resizing, we get an embedding
  \m{D \to L}, because embeddings are closed under composition and
  equivalences are embeddings.  Hence \m{D} is a retract of \m{L} in
  an unspecified way by the injectivity of~\m{D}, by extending its
  identity.  But \m{L}, being equivalent to a free algebra, is
  algebraically injective, and hence, being a retract of \m{L} in an
  unspecified way, \m{D} is algebraically injective in an unspecified
  way, because retracts of algebraically injectives are algebraically
  injective, by the functoriality of truncation on objects.
\end{proof}

As an immediate consequence, by reduction to the above results about algebraic
injectivity, we have the following corollary.
\begin{theorem}
  Under \m{\Omega}-resizing and propositional resizing, if a type
  \m{D} in a universe \m{\T} is \m{\T,\T}-injective , then it is also
  \m{\U,\V}-injective for any universes \m{\U} and \m{\V}.
\end{theorem}
\begin{proof}
  The type \m{D} is algebraically \m{\T,\T}-injective in an
  unspecified way, and so by functoriality of truncation on objects
  and algebraic injective resizing, it is algebraically
  \m{\U,\V}-injective in an unspecified way, and hence it is
  \m{\U,\V}-injective.
\end{proof}
At the time of writing, we are not able to establish the converse. In
particular, we don't have the analogue of~\ref{universe:independence}.

\section{The equivalence of  excluded middle with the (algebraic) injectivity of all pointed types}

Algebraic flabbiness can also be applied to show that all pointed
types are (algebraically) injective if and only if excluded middle
holds, where for injectivity resizing is needed as an assumption, but
for algebraic injectivity it is not.

The decidability of a type \m{X} is defined to be the assertion \m{X +
  (X \to \Zero)}, which says that we can exhibit a point of \m{X} or
else tell that \m{X} is empty.
The principle of excluded middle in univalent mathematics, for the
universe \m{\U}, is taken to mean that all subsingleton types in
\m{\U} are decidable:
\M{
  \EM_\U \eqdef \Pi (P : \U), \text{\m{P} is a subsingleton \m{\to P + (P \to \Zero)}.}
}
As discussed in the introduction, we are not assuming or rejecting
this principle, which is independent of the other axioms. Notice
that, in the presence of function extensionality, this principle is
a subsingleton, because products of subsingletons are subsingletons and
because \m{P + (P \to \Zero)} is a subsingleton for any
subsingleton \m{P}. So in the following we get data out of a proposition.

\begin{lemma}
  If excluded middle holds in the universe \m{\U}, then every pointed
  type \m{D} in any universe \m{\W} is algebraically \m{\U}-flabby.
\end{lemma}
\begin{proof}
  Let \m{d} be the given point of \m{D} and \m{f : P \to D} be a function with
  subsingleton domain. If we have a point \m{p : P}, then we can take
  \m{f \, p} as the flabbiness witness. Otherwise, if \m{P \to \Zero},
  we can take \m{d} as the flabbiness witness.
\end{proof}
\noindent
For the converse, we use the following.
\begin{lemma}
  If the type \m{P + (P \to \Zero) + \One} is algebraically
  \m{\W}-flabby for a given subsingleton \m{P} in a universe \m{\W},
  then \m{P} is decidable.
\end{lemma}
\begin{proof}
  Denote by \m{D} the type \m{P + (P \to \Zero) + \One} and let \m{f :
    P + (P \to \Zero) \to D} be the inclusion. Because \m{P + (P \to
    \Zero)} is a subsingleton, the algebraic flabbiness of \m{D} gives
  \m{d : D} with \m{d = f \, z} for all \m{z : P + (P \to \Zero)}.
  Now, by definition of binary sum, \m{d} must be in one of the three
  components of the sum that defines~\m{D}.  If it were in the third
  component, namely \m{\One}, then \m{P} couldn't hold, because if it
  did we would have \m{p:P} and hence, omitting the inclusions into
  sums, and considering \m{z=p}, we would have, \m{d = f p = p},
  because \m{f} is the inclusion, which is not in the \m{\One}
  component. But also \m{P \to \Zero} couldn't hold, because if it did
  we would have \m{\phi:P \to \Zero} and hence, again omitting the
  inclusion, and considering \m{z=\phi}, we would have \m{d = f \,
    \phi = \phi}, which again is not in the \m{\One} component. But
  it is impossible for both \m{P} and \m{P \to \Zero} to fail, because
  this would mean that we would have functions \m{P \to \Zero} (the
  failure of \m{P}) and \m{(P \to \Zero) \to \Zero} (the failure of
  \m{P \to \Zero}), and so we could apply the second function to the
  first to get a point of the empty type, which is not
  available. Therefore \m{d} can't be in the third component, and so
  it must be in the first or the second, which means that \m{P} is
  decidable.
\end{proof}
\noindent From this we immediately conclude the following:
\begin{lemma}
  If all pointed types in a universe \m{\W} are algebraically \m{\W}-flabby, then excluded middle holds in~\m{\W}.
\end{lemma}
\noindent
And then we have the same situation for algebraically injective types,
by reduction to algebraic flabbiness:
\begin{lemma}
  If excluded middle holds in the universe \m{\U \sqcup \V}, then any
  pointed type \m{D} in any universe \m{\W} is algebraically
  \m{\U,\V}-injective.
\end{lemma}
\noindent Putting this together with some universe specializations, we
have the following construction.
\begin{theorem} 
  All pointed types in a universe \m{\U} are algebraically \m{\U,\U}-injective if and only if excluded middle holds in~\m{\U}.
\end{theorem}

\noindent
And we have a similar situation with injective types.
\begin{lemma}
  If excluded middle holds, then every inhabited type of any universe is
  injective with respect to any two universes.
\end{lemma}
\begin{proof}
  Because excluded middle gives algebraic injectivity, which in turn gives
  injectivity.
\end{proof}
\noindent Without resizing, we have the following.
\begin{lemma}
  If every inhabited type \m{D:\W} is \m{\W,\W^+}-injective, then
  excluded middle holds in the universe \m{\W}.
\end{lemma}
\begin{proof}
  Given a proposition \m{P}, we have that the type \m{D \eqdef P + (P
    \to \Zero) + \One_{\W}} is injective by the assumption. Hence it
  is algebraically injective in an unspecified way by
  Proposition~\ref{worse}. And so it is algebraically flabby in an
  unspecified way.  By the lemma, \m{P} is decidable in an unspecified
  way, but then it is decidable because the decidability of a
  proposition is a proposition.
\end{proof}
\noindent With resizing we can do better:
\begin{lemma}
  Under \m{\Omega}-resizing, if every inhabited type in a universe \m{\U} is
  \m{\U,\U}-injective, then excluded middle holds in \m{\U}.
\end{lemma}
\begin{proof}
  Given a proposition \m{P}, we have that the type \m{D \eqdef P + (P
    \to \Zero) + \One_{\U}} is injective by the assumption. Hence it
  is injective in an unspecified way by
  Theorem~\ref{better}. And so it is algebraically flabby in an
  unspecified way.  By the lemma, \m{P} is decidable in an unspecified
  way, and hence decidable.
\end{proof}
\begin{theorem} 
  Under \m{\Omega}-resizing, all inhabited types in a universe \m{\U} are
  \m{\U,\U}-injective if and only if excluded middles
  holds in~\m{\U}.
\end{theorem}
\noindent It would be interesting to get rid of the resizing assumption, which,
as we have seen, is not needed for the equivalence of the algebraic
injectivity of all pointed types with excluded middle.

\bibliographystyle{plain}
\bibliography{references}

\end{document}